\documentclass{amsart}
\usepackage[T1]{fontenc}
\usepackage{enumerate,graphicx,tabls,bm,amssymb,xcolor}
  \definecolor{refkey}{rgb}{0.2,0,0}%
  \definecolor{labelkey}{rgb}{0.5,0.5,1}%
\usepackage[centering, a4paper]{geometry}

 \usepackage{mathptmx} 
 \usepackage{calrsfs}

\usepackage{graphicx,pst-plot}

  \definecolor{light-red}{rgb}{0.85,0,0}%

\def\ss#1{\ifnum#1=1\let\next=\sone\else\let\next=\stwo\fi\next}
\def\tt#1{\ifnum#1=1\let\next=\tone\else\let\next=\ttwo\fi\next}
\def\ttm#1{\ifnum#1=2\let\next=\tone\else\let\next=\ttwo\fi\next}
\def\sone{{v+s}}
\def\stwo{{v-s}}
\def\tone{{u+t}}
\def\ttwo{{u-t}}

\newsavebox{\savepar}

\newtheorem{thm}{\global\def\endclaim{\end{thm}}Theorem}
   
   \newtheorem{lemma}[thm]{\global\def\endclaim{\end{lemma}}Lemma}
   \newtheorem{prop}[thm]{\global\def\endclaim{\end{prop}}Proposition}
   \newtheorem{rem}[thm]{\global\def\endclaim{\end{rem}}Remark}
\theoremstyle{definition}

\numberwithin{equation}{section}

\title[On multilinear polynomials in four variables evaluated on matrices]
    {On multilinear polynomials in four variables evaluated on matrices}

\author{David Buzinski}
\address{Case Western Reserve University}
\email{dab197@case.edu}
\author{Robin Winstanley}
\address{University of Washington}
\email{winstanr@uw.edu}
\thanks{$^*$Supported in part by National Science Foundation, grant \#1156798.}


\begin{document}

\begin{abstract}
Let $K$ be an algebraically closed field of characteristic $0$ and let $M_n(K)$, $n \ge 3$, be the matrix ring over $K$.
We will show that the image of any multilinear polynomial in four variables evaluated on $M_n(K)$ contains all matrices of trace $0$.
\end{abstract}

\maketitle

\section{Introduction}
In 1936 Shoda \cite{sh} proved that over a field of characteristic $0$, every matrix of trace $0$ can be represented as a commutator $AB-BA$ of two matrices. In 1957 Albert and Muckenhoupt \cite{am} extended the result to all fields. An old open question (Problem 1.98 in the Dniester Notebook, Fourth Edition, communicated by Lvov, also attributed to Kaplansky, see \cite{bmr1}) asks: "Let $f$ be a multilinear polynomial over a field $K$. Is the set of values of $f$ on the matrix algebra $M_n(K)$ a vector space?" There were few developments in this problem until recently when Kanel-Belov, Malev and Rowen solved the problem for $2\times2$ matrices \cite{bmr1} and made important progress towards the $3\times3$ case \cite{bmr2}.

Shoda's theorem can be reformulated in the form that the set of values of the polynomial $f(x,y)=xy-yx$ on the algebra of matrices contains all matrices of trace 0. In 2013 Shoda's result was generalized by Mesyan \cite{m} for multilinear polynomials of degree $3$ and by \v Spenko \cite{s} for Lie polynomials of degree $\le 4$.

Throughout the paper, we will use $[M_n(K),M_n(K)]$ to denote the $K$-subspace of $M_n(K)$ consisting of the matrices of trace zero and $f(M_n(K))$ to denote the set of values of a polynomial $f$ on $M_n(K)$. By a multilinear polynomial we will understand the polynomial in noncommutative variables and linear in each variable. Our main goal is to prove the following result:

\begin{thm}\label{T1}
Let $n \geq 3$ be an integer, $K$ an algebraically closed field of characteristic $0$, and $f \in K\langle x_1, x_2, x_3, x_4 \rangle$ any nonzero multilinear polynomial. Then $[M_{n}(K), M_{n}(K)]\subseteq f(M_{n}(K))$.
\end{thm}

The condition on $n \ge 3$ is necessary as, for instance, the theorem fails for $n=2$ and $f(x_1,x_2,x_3,x_4)=[x_1,x_2][x_3,x_4]+[x_3,x_4][x_1,x_2]$ which is a famous example of a central polynomial.

Our theorem confirms the following conjecture due to Mesyan \cite[Conjecture 11]{m} for the case $m=4$:

{\bf Conjecture.} Let $K$ be a field, $n \ge 2$ and $m \ge 1$ integers, and $f(x_1,\ldots,x_m)$ a nonzero multilinear polynomial in $K \langle x_1,\ldots,x_m \rangle$. If $n \ge m-1$, then $[M_{n}(K), M_{n}(K)]\subseteq f(M_{n}(K))$.

This problem remains open for $m>4$.
\section{The results}

We will start with several auxiliary results. By $e_{i,j}$ we denote a standard matrix unit, that is an $n \times n$ matrix with $1$ in the $i$'th row and $j$'th column and zeros elsewhere.

\begin{lemma}\label{L1}
Let $K$ be a field and let $a_{i,j} \in K$ be any elements with $\sum_{i=1}^{n}a_{i,i}=0$.
Let $A=\sum_{i=1}^{n-1}e_{i,i+1} \in M_n(K)$. Then there exists a $B \in M_n(K)$ such that
\begin{equation}\label{E1}
[A,B]= \sum_{i=1}^{n}a_{i,i}e_{i,i}+\sum_{i=1}^{n-1}a_{i,i+1}e_{i,i+1}.
\end{equation}
\end{lemma}

\begin{proof}
We will be looking for a matrix $B$ in the form $\sum_{i=1}^{n}b_{i,i}e_{i,i}+\sum_{i=1}^{n-1}b_{i+1,i}e_{i+1,i}$.
Then
\begin{equation}\label{E2}
[A,B]=b_{2,1}e_{1,1}+\sum_{i=2}^{n-1}(b_{i+1,i}-b_{i,i-1})e_{i,i} -b_{n,n-1}e_{n,n}
+\sum_{i=1}^{n-1}(b_{i+1,i+1}-b_{i,i})e_{i,i+1}.
\end{equation}
Comparing (\ref{E1}) and (\ref{E2}), we arrive at two systems of linear equations in indeterminates $b_{i,j}$'s.
\begin{equation}\label{E3}\begin{array}{lcl} b_{2,1} &=& a_{1,1} \\ b_{3,2}-b_{2,1} &=& a_{2,2}\\ &\vdots& \\ b_{n,n-1}-b_{n-1,n-2}&=&a_{n-1,n-1}\\b_{n,n-1}&=&\sum_{i=1}^{n-1}a_{i,i},
\end{array}
\end{equation}
and
\begin{equation}\label{E4}\begin{array}{lcl} b_{2,2}-b_{1,1}&=&a_{1,2}\\b_{3,3}-b_{2,2}&=&a_{2,3}\\ &\vdots& \\ b_{n,n}-b_{n-1,n-1}&=&a_{n-1,n}.
\end{array}
\end{equation}
Clearly, both systems have solutions. For example, the system (\ref{E3}) has solutions:
$$\begin{array}{lcl}b_{2,1}&=&a_{1,1}\\b_{3,2}&=&a_{1,1}+a_{2,2}\\& \vdots &\\b_{n,n-1} &=& a_{1,1}+a_{2,2}+\ldots+a_{n-1,n-1},
\end{array}$$
and the system (\ref{E4}) has solutions:
$$\begin{array}{lcl}b_{1,1}&=&0\\b_{2,2}&=&a_{1,2}\\b_{3,3}&=&a_{1,2}+a_{2,3}\\b_{4,4}&=&a_{1,2}+a_{2,3}+a_{3,4}\\&\vdots&\\b_{n,n}&=&a_{1,2}+a_{2,3}+\ldots + a_{n-1,n}.
\end{array}$$
This completes the lemma.
\end{proof}

We will need the following well known facts about the images of multilinear polynomials that can be found, for example, in \cite[Lemma 6 and Corollary 8]{m}.
\begin{rem}\label{R1}
    Let $K$ be a field, $m$ a positive integer, and $$f(x_1,\ldots,x_m)=\sum_{\sigma \in S_m}a_{\sigma}x_{\sigma(1)}x_{\sigma(2)}\ldots x_{\sigma(m)} \in K\langle x_1,\ldots,x_m \rangle$$ a multilinear polynomial, where $S_m$ is the permutation group on $m$ elements.
    \begin{enumerate}
        \item[(i)] If  $\sum_{\sigma \in S_m}a_{\sigma} \neq 0$, then $f(M_n(K))=M_n(K)$.
        \item[(ii)] For any invertible matrix $A \in M_n(K)$ and any matrices $B_1,\ldots,B_m \in M_n(K)$ we have $$Af(B_1,\ldots,B_m)A^{-1}=f(AB_1A^{-1},\ldots,AB_mA^{-1}).$$
    \end{enumerate}

\end{rem}

We are ready to prove our next lemma.
\begin{lemma}\label{L2}
Let $K$ be an algebraically closed field of characteristic $0$. Let $\lambda$ be an element of $K$ and $\lambda\ne -1$. Then for every matrix $D \in M_n(K)$ of trace $0$ with $n \ge 3$, there exist $A,B,C \in M_n(K)$ such that $D=[A,B][A,C]+\lambda[A,C][A,B]$.
\end{lemma}

\begin{proof}
Any matrix with coefficients in an algebraically closed field is similar to its Jordan canonical form, so by Remark~\ref{R1} it is enough to show that every matrix $D$ of trace $0$ in the form $\sum_{i=1}^{n}d_{i,i}e_{i,i}+\sum_{i=1}^{n-1}d_{i,i+1}e_{i,i+1}$ can be represented as $[A,B][A,C]+\lambda[A,C][A,B]$. We will consider two cases.

Case 1:
Suppose $\lambda \neq \frac{1}{n-1}$. Then by Lemma~\ref{L1} if we let $A=\sum_{i=1}^{n-1}e_{i,i+1}$, then there exist $B,C \in M_n(K)$ such that $[A,B]=\sum_{i=1}^{n-1}e_{i,i}-(n-1)e_{n,n}$ and $[A,C]=\sum_{i=1}^{n}b_{i,i}e_{i,i}+\sum_{i=1}^{n-1}b_{i,i+1}e_{i,i+1}$ where $b_{n,n}=-\sum_{i=1}^{n-1}b_{i,i}$. Then we obtain
\begin{multline*}
[A,B][A,C]+\lambda[A,C][A,B]=\\(\lambda+1)\sum_{i=1}^{n-1}b_{i,i}e_{i,i}-(n-1)(\lambda+1)b_{n,n}e_{n,n}+(\lambda+1)\sum_{i=1}^{n-2}b_{i,i+1}e_{i,i+1}+(1-(n-1)\lambda)b_{n-1,n}e_{n-1,n}.
\end{multline*}
Like in Lemma~\ref{L1}, we arrive at two systems of linear equations in indeterminates $b_{i,j}$'s.
$$\begin{array}{lcl}(\lambda+1)b_{1,1}&=&d_{1,1}\\(\lambda+1)b_{2,2}&=&d_{2,2}\\&\vdots&\\(\lambda+1)b_{n-1,n-1}&=&d_{n-1,n-1}\\-(n-1)(\lambda+1)b_{n,n}&=&d_{n,n},
\end{array}$$
(we included the last equation just for consistency, it is actually the sum of all previous equations multiplied by $-1$),
and
$$\begin{array}{lcl}(\lambda+1)b_{1,2}&=&d_{1,2}\\(\lambda+1)b_{2,3}&=&d_{2,3}\\&\vdots&\\(\lambda+1)b_{n-2,n-1}&=&d_{n-2,n-1}\\(1-(n-1)\lambda)b_{n-1,n}&=&d_{n-1,n}.
\end{array}$$
Since $\lambda \neq -1$ and $\lambda \neq \frac{1}{n-1}$, obviously both systems have solutions.

Case 2:
Suppose $\lambda = \frac{1}{n-1}$. Then we take $[A,B]=\sum_{i=1}^{n-2}e_{i,i}+2e_{n-1,n-1}-ne_{n,n}$ and $[A,C]=\sum_{i=1}^{n}b_{i,i}e_{i,i}+\sum_{i=1}^{n-1}b_{i,i+1}e_{i,i+1}$ where $b_{n,n}=-\sum_{i=1}^{n-1}b_{i,i}$. Then we obtain
\begin{multline*}
[A,B][A,C]+\lambda[A,C][A,B]=(\lambda+1)\sum_{i=1}^{n-2}b_{i,i}e_{i,i}+2(\lambda+1)b_{n-1,n-1}e_{n-1,n-1}-n(\lambda+1)b_{n,n}e_{n,n}\\+(\lambda+1)\sum_{i=1}^{n-3}b_{i,i+1}e_{i,i+1}+(1+2\lambda)b_{n-2,n-1}e_{n-2,n-1}+(2-n\lambda)b_{n-1,n}e_{n-1,n}.
\end{multline*}
Again we arrive at two systems of linear equations
$$\begin{array}{lcl}(\lambda+1)b_{1,1}&=&d_{1,1}\\(\lambda+1)b_{2,2}&=&d_{2,2}\\&\vdots&\\(\lambda+1)b_{n-2,n-2}&=&d_{n-2,n-2}\\2(\lambda+1)b_{n-1,n-1}&=&d_{n-1,n-1}\\-n(\lambda+1)b_{n,n}&=&d_{n,n},
\end{array}$$
and
$$\begin{array}{lcl}(\lambda+1)b_{1,2}&=&d_{1,2}\\(\lambda+1)b_{2,3}&=&d_{2,3}\\&\vdots&\\(\lambda+1)b_{n-3,n-2}&=&d_{n-3,n-2}\\(2\lambda+1)b_{n-2,n-1}&=&d_{n-2,n-1}\\(2-n\lambda)b_{n-1,n}&=&d_{n-1,n}.
\end{array}$$
Since $\lambda =\frac{1}{n-1}$, both systems have solutions.

Therefore, we can get the desired representation and the proof is complete.
\end{proof}

The next result follows from the proof of \cite[Lemma 7.4]{s}:
\begin{lemma}\label{L2A}
Let $K$ be an algebraically closed field of characteristic $0$. Then for every matrix $D \in M_n(K)$ of trace $0$ with $n \ge 3$, there exist $A,B,C \in M_n(K)$ such that $D=[A,B][A,C]-[A,C][A,B]$.
\end{lemma}

We continue with
\begin{lemma} \label{L3}
    Let $n \ge 3$ be an integer. Let $K$ be a field of characteristic $0$ and let $a_{ij} \in K$ be any elements with $\sum_{i=1}^{n}a_{i,i}=0$. Then there exist $A,B,C \in M_n(K)$ such that $[A,[[A,B],[A,C]]]=\sum_{i=1}^{n}a_{i,i}e_{i,i}+\sum_{i=1}^{n-1}a_{i,i+1}e_{i,i+1}$.
\end{lemma}
\begin{proof}
Let $A \in M_{n}(K)$ be a matrix of the form $\sum_{i=1}^{n-1}e_{i,i+1}$. Then let $B,C \in M_{n}(K)$ such that
\begin{equation}\label{E5}
B=\sum_{i=1}^{n}b_{i,i}e_{i,i}+\sum_{i=1}^{n-1}b_{i+1,i}e_{i+1,i}
\end{equation}
 with $\sum_{i=1}^{n}b_{i,i}=0$, and $C=\sum_{i=3}^{n}(i-2)e_{i,i-2}$. Then $[A,C]=\sum_{i=1}^{n-2}e_{i+1,i}-(n-2)e_{n,n-1}$ and $[A,B]=\sum_{i=1}^{n-1}(b_{i+1,i+1}-b_{i,i})e_{i,i+1}+b_{2,1}e_{1,1}+\sum_{i=2}^{n-1}(b_{i+1,i}-b_{i,i-1})e_{i,i}-b_{n,n-1}e_{n,n}$. Note that for any $c_{i,j}$ with $\sum_{i=1}^{n}c_{i,i}=0$, we can get $[A,B]$ in the form $[A,B]=\sum_{i=1}^{n}c_{i,i}e_{i,i}+\sum_{i=1}^{n-1}c_{i,i+1}e_{i,i+1}$. Indeed, the two systems of linear equations in variables $b_{i,j}$
$$\begin{array}{lcl}b_{2,2}-b_{1,1}&=&c_{1,2}\\b_{3,3}-b_{2,2}&=&c_{2,3}\\&\vdots&\\b_{n,n}-b_{n-1,n-1}&=&c_{n-1,n},
\end{array}$$
and
$$\begin{array}{lcl}b_{2,1}&=&c_{1,1}\\b_{3,2}-b_{2,1}&=&c_{2,2}\\b_{4,3}-b_{3,2}&=&c_{3,3}\\&\vdots&\\b_{n,n-1}-b_{n-1,n-2}&=&c_{n-1,n-1}\\-b_{n,n-1}&=&c_{n,n}
\end{array}$$
have solutions satisfying $\sum_{i=1}^{n}b_{i,i}=0$. Then for $[A,B]=\sum_{i=1}^{n}c_{i,i}e_{i,i}+\sum_{i=1}^{n-1}c_{i,i+1}e_{i,i+1}$ and $[A,C]=\sum_{i=1}^{n-2}e_{i+1,i}-(n-2)e_{n,n-1}$ we get
\begin{equation*}
\begin{split}
[[A,B],[A,C]]=c_{1,2}e_{1,1}+\sum_{i=2}^{n-2}(c_{i,i+1}-c_{i-1,i})e_{i,i}+(-(n-2)c_{n-1,n}-c_{n-2,n-1})e_{n-1,n-1}\\+(n-2)c_{n-1,n}e_{n,n}+\sum_{i=1}^{n-2}(c_{i+1,i+1}-c_{i,i})e_{i+1,i}-(n-2)(c_{n,n}-c_{n-1,n-1})e_{n,n-1}.
\end{split}
\end{equation*}
Observe that for any $d_{i,j}$ with $\sum_{i=1}^{n}d_{i,i}=0$, the systems of linear equations in variables $c_{i,j}$
$$\begin{array}{lcl}c_{2,2}-c_{1,1}&=&d_{2,1}\\c_{3,3}-c_{2,2}&=&d_{3,2}\\&\vdots&\\c_{n-1,n-1}-c_{n-2,n-2}&=&d_{n-1,n-2}\\-(n-2)(c_{n,n}-c_{n-1,n-1})&=&d_{n,n-1},
\end{array}$$
and
$$\begin{array}{lcl}c_{1,2}&=&d_{1,1}\\c_{2,3}-c_{1,2}&=&d_{2,2}\\c_{3,4}-c_{2,3}&=&d_{3,3}\\&\vdots&\\c_{n-2,n-1}-c_{n-3,n-2}&=&d_{n-2,n-2}\\-(n-2)c_{n-1,n}-c_{n-2,n-1}&=&d_{n-1,n-1}\\(n-2)c_{n-1,n}&=&d_{n,n}
\end{array}$$
have solutions satisfying $\sum_{i=1}^{n}c_{i,i}=0$. Now we have $D=[[A,B],[A,C]]=\sum_{i=1}^{n}d_{i,i}e_{i,i}+\sum_{i=1}^{n-1}d_{i+1,i}e_{i+1,i}$ with $\sum_{i=1}^{n}d_{i,i}=0$. Observe that the matrix $D$ is of the same form as matrix $B$ in equation (\ref{E5}) so $[A,D]$ can be represented in the form $\sum_{i=1}^{n}a_{i,i}e_{i,i}+\sum_{i=1}^{n-1}a_{i,i+1}e_{i,i+1}$.
\end{proof}

Before we start the proof of our main result, we need to handle one special case.

\begin{prop} \label{P1}
    Let $n\ge 3$ be an integer and let $K$ be an algebraically closed field of characteristic $0$. If
    \begin{multline*}
    f(x_1,x_2,x_3,x_4)=\\ [x_1,x_2][x_3,x_4]+[x_3,x_4][x_1,x_2]+
    [x_2,x_3][x_1,x_4]+[x_1,x_4][x_2,x_3]-[x_1,x_3][x_2,x_4]-[x_2,x_4][x_1,x_3],
    \end{multline*}
     then $[M_{n}(K), M_{n}(K)]\subseteq f(M_{n}(K))$.
\end{prop}
\begin{proof}
Let $A,B,C\in M_{n}(K)$. Then
    \begin{equation*}
        \begin{split}
            f(A,A^2,B,C)&=[A,A^2][B,C]+[B,C][A,A^2]+[A^2,B][A,C]
            \\&+[A,C][A^2,B]-[A,B][A^2,C]-[A^2,C][A,B]\\
            &=[A^2,B][A,C]+[A,C][A^2,B]-[A,B][A^2,C]\\
            &-[A^2,C][A,B]\\
            &=[A,[[A,B],[A,C]]].
        \end{split}
    \end{equation*}
By Lemma~\ref{L3}, there exist $A,B,C\in M_{n}(K)$ such that $f(A,A^2,B,C)$ can be represented as $\sum_{i=1}^{n}a_{i,i}e_{i,i}+\sum_{i=1}^{n-1}a_{i,i+1}e_{i,i+1}$ with $\sum_{i=1}^{n}a_{i,i}=0$. This includes the Jordan canonical forms of all matrices with trace zero. By Remark~\ref{R1}, we obtain $[M_{n}(K), M_{n}(K)]\subseteq f(M_{n}(K))$.
\end{proof}

We are ready to prove our main result.

{\bf Proof of Theorem~\ref{T1}.}
We can write any multilinear polynomial $f \in K\langle x_1, x_2, x_3, x_4 \rangle$ in the form
    \begin{equation*}
        f(x_1,x_2,x_3,x_4) = \sum\limits_{\sigma\in S_4}a_{\sigma}x_{\sigma(1)}x_{\sigma(2)}x_{\sigma(3)}x_{\sigma(4)},
    \end{equation*}
    where $S_4$ is the permutation group on four elements. If $\sum\limits_{\sigma\in S_4}a_{\sigma}\neq0$, then by Remark~\ref{R1}, $f(M_{n}(K))$ contains all matrices, so in particular $[M_{n}(K), M_{n}(K)]\subseteq f(M_{n}(K))$.

     Therefore, we can assume $\sum\limits_{\sigma\in S_4}a_{\sigma}=0$.

    If the partial derivative of $f$ with respect to some $x_i$ is nonzero, then we can set $x_i$ equal to the identity matrix $1$ and the situation reduces to the three variable case covered in \cite[Theorem 13]{m}.

    Therefore we can assume all partial derivatives of $f$ are zero, that is
    \begin{equation}
        f(1,x_2,x_3,x_4)=f(x_1,1,x_3,x_4)=f(x_1,x_2,1,x_4)=f(x_1,x_2,x_3,1)=0.
        \nonumber
    \end{equation}
    By Falk's Theorem \cite{f} (see also \cite[Theorem 1.1]{j} for the English version) any such polynomial is a product of (iterated) commutators. It means that our polynomial can be written in the form
    \begin{equation*}
        \begin{split}
        f(x_1,x_2,x_3,x_4)=L(x_1,x_2,x_3,x_4)+c_{1234}[x_1,x_2][x_3,x_4]+c_{1324}[x_1,x_3][x_2,x_4]+
        c_{1423}[x_1,x_4][x_2,x_3]\\
        +c_{2314}[x_2,x_3][x_1,x_4]+c_{2413}[x_2,x_4][x_1,x_3]
        +c_{3412}[x_3,x_4][x_1,x_2],
        \end{split}
    \end{equation*}
    where $L(x_1,x_2,x_3,x_4)$ is a Lie polynomial in four variables. According to \cite[Theorem 3.1]{h}, every such polynomial can be written as follows:
    \begin{multline*}
    L(x_1,x_2,x_3,x_4)=z_1[[[x_2,x_1],x_3],x_4]+z_2[[[x_3,x_1],x_2],x_4]+z_3[[[x_4,x_1],x_2],x_3]\\+z_4[[x_4,x_1],[x_3,x_2]]+z_5[[x_4,x_2],[x_3,x_1]]+z_6[[x_4,x_3],[x_2,x_1]].
    \end{multline*}
    Observe that the last three terms can be written as a linear combination of products of commutators. Thus without loss of generality we may assume that $f(x_1,x_2,x_3,x_4)$ is of the form
    \begin{multline*}
    f(x_1,x_2,x_3,x_4)=z_1[[[x_2,x_1],x_3],x_4]+z_2[[[x_3,x_1],x_2],x_4]+z_3[[[x_4,x_1],x_2],x_3]\\+c_{1234}[x_1,x_2][x_3,x_4]+c_{1324}[x_1,x_3][x_2,x_4]+
        c_{1423}[x_1,x_4][x_2,x_3]\\+
        c_{2314}[x_2,x_3][x_1,x_4]+c_{2413}[x_2,x_4][x_1,x_3]
        +c_{3412}[x_3,x_4][x_1,x_2].
    \end{multline*}
    Suppose that for some $i=1,2,3$, $z_i \neq 0$. Say, let $z_1 \neq 0$. Arguing as in the proof of \cite[Lemma 7.4]{s}, we take $x_1=x_3=x_4=S$, where $S$ is a diagonal matrix with distinct diagonal entries. By \cite[Lemma 1.2]{ar}, $f(S,x_2,S,S)$ consists of all matrices with only zeros on the main diagonal. By \cite[Proposition 1.8]{ar}, every matrix of trace zeros is similar to a matrix with only zero on the main diagonal, so by Remark~\ref{R1}, we obtain $[M_{n}(K), M_{n}(K)]\subseteq f(M_{n}(K))$. The cases when $z_2 \neq 0$ or $z_3 \neq 0$ can be treated similarly. Now we can assume that $z_1=z_2=z_3=0$ and $f$ is of the form
    \begin{multline*}
    f(x_1,x_2,x_3,x_4)=c_{1234}[x_1,x_2][x_3,x_4]+c_{1324}[x_1,x_3][x_2,x_4]+
        c_{1423}[x_1,x_4][x_2,x_3]\\+
        c_{2314}[x_2,x_3][x_1,x_4]+c_{2413}[x_2,x_4][x_1,x_3]
        +c_{3412}[x_3,x_4][x_1,x_2].
    \end{multline*}

    Consider two cases.
    Case 1: Assume $c_{1234}=c_{2314}=c_{3412}=c_{1423}=-c_{1324}=-c_{2413}$.
    Then
    \begin{multline*}
    f(x_1,x_2,x_3,x_4)=c_{1234}([x_1,x_2][x_3,x_4]+[x_3,x_4][x_1,x_2]+
    [x_2,x_3][x_1,x_4]+[x_1,x_4][x_2,x_3]\\-[x_1,x_3][x_2,x_4]-[x_2,x_4][x_1,x_3]).
    \end{multline*}
     Applying Proposition~\ref{P1}, we can conclude $[M_{n}(K), M_{n}(K)]\subseteq f(M_{n}(K))$.

    Case 2: Assume at least one of the following $c_{1234}=c_{2314}=c_{3412}=c_{1423}=-c_{1324}=-c_{2413}$ does not hold. For any $A,B,C \in M_{n}(K)$,
    at least one of the following expressions is not zero:
    $$f(A,A,B,C)=(c_{1324}+c_{2314})[A,B][A,C]+(c_{1423}+c_{2413})[A,C][A,B],$$
    $$f(A,B,A,C)=(c_{1234}-c_{2314})[A,B][A,C]+(c_{3412}-c_{1423})[A,C][A,B],$$
    $$f(A,B,C,A)=(-c_{1234}-c_{2413})[A,B][A,C]+(-c_{1324}-c_{3412})[A,C][A,B],$$
    $$f(B,A,A,C)=(-c_{1234}-c_{1324})[A,B][A,C]+(-c_{2413}-c_{3412})[A,C][A,B],$$
    $$f(B,A,C,A)=(-c_{1423}+c_{1234})[A,B][A,C]+(c_{3412}-c_{2314})[A,C][A,B],$$
    $$f(B,C,A,A)=(c_{1324}+c_{1423})[A,B][A,C]+(c_{2314}+c_{2413})[A,C][A,B].$$

    The problem is now reduced to the polynomial $[A,B][A,C]+\lambda[A,C][A,B]$. If $\lambda \neq -1$, then by Lemma~\ref{L2}, the image of $f$ contains all matrices of trace zero and our theorem is complete. If $\lambda = -1$, then the desired result follows from Lemma~\ref{L2A}. The proof is complete.

{\bf Acknowledgement.} We would like to thank our advisor Dr. Mikhail Chebotar for his help and encouragement. We would like to thank the Department of Mathematical Sciences at Kent State University for its hospitality. We are grateful to the referee for her/his useful suggestions.

\end{document}